\documentclass[12pt]{amsart}
\usepackage{a4wide,amsfonts,graphicx,
amssymb,stmaryrd,amscd,amsmath,latexsym,amsbsy,xypic}
\xyoption{all}
\usepackage{marginnote}

\theoremstyle{plain}

\newtheorem{thm}{Theorem}

\newtheorem{prop}[thm]{Proposition}

\newtheorem{defi}[thm]{Definition}
\newtheorem{lem}[thm]{Lemma}
\newtheorem{cor}[thm]{Corollary}

\theoremstyle{remark}
\newtheorem{rema}[thm]{Remark}
\newtheorem{eg}[thm]{Example}

\title{Some remarks on non-symmetric interpolation Macdonald polynomials}
\author{Siddhartha Sahi \& Jasper Stokman}
\address{S. Sahi, Department of Mathematics, Rutgers University, 110 Frelinghhuysen Rd,
Piscataway, NJ 08854-8019, USA.}
\email{sahi@math.rutgers.edu}
\address{J. Stokman, KdV Institute for Mathematics, University of Amsterdam,
Science Park 105-107, 1098 XG Amsterdam, The Netherlands.}
\email{j.v.stokman@uva.nl}
\begin{document}
\begin{abstract}
We provide elementary identities relating the three known types of non-symmetric interpolation Macdonald polynomials. In addition we derive a duality for non-symmetric interpolation Macdonald polynomials. We consider some applications of these results, in particular to binomial formulas involving non-symmetric interpolation Macdonald polynomials.
\end{abstract}
\vspace{2em}
\keywords{}
\maketitle

\section{Introduction}
The symmetric interpolation Macdonald polynomials $R_{\lambda}(x;q,t) =R_{\lambda}(x_{1},\ldots ,x_{n};q,t) $ form a distinguished inhomogeneous
basis for the algebra of symmetric polynomials in $n$ variables over the
field $\mathbb{F}:=\mathbb{Q}(q,t)$. They were first introduced in \cite{Kn,
S-int}, building on joint work by one of the authors with F. Knop \cite{KS1}
and earlier work with B. Kostant \cite{KoS1, KoS2, S-inv}. These polynomials
are indexed by the set of partitions with at most $n$ parts 
\begin{equation*}
\mathcal{P}_{n}:=\left\{\, \lambda \in \mathbb{Z}^{n}\,\,\, | \,\,\, \lambda _{1}\geq \lambda
_{2}\geq \cdots \geq \lambda _{n}\geq 0 \,\right\} .
\end{equation*}
For a partition $\mu \in \mathcal{P}_{n}$ we define $\left\vert \mu
\right\vert =\mu _{1}+\cdots +\mu _{n}$ and write 
\begin{equation*}
\overline{\mu}=\left( q^{\mu _{1}}\tau _{1},\ldots ,q^{\mu _{n}}\tau
_{n}\right) \text{ where } \tau:=(\tau_1,\ldots,\tau_n) \text{ with }
\tau _{i}:=t^{1-i}.
\end{equation*}
Then $R_{\lambda}(x) =R_{\lambda}(x;q,t)$ is, up to normalization, characterized as the unique nonzero symmetric polynomial of degree at most $\left\vert \lambda \right\vert$ 
satisfying the vanishing conditions
\begin{equation*}
R_{\lambda}(\overline{\mu}) =0\text{ for } \mu\in\mathcal{P}_n \text{ such that } \left\vert \mu
\right\vert \leq \left\vert \lambda \right\vert ,\text{ }\mu \neq \lambda.
\end{equation*}
The normalization is fixed by requiring that the coefficient of $x^\lambda:=x_{1}^{\lambda_{1}}\cdots x_{n}^{\lambda _{n}}$ in the monomial expansion of $R_\lambda(x)$ is $1$.
In spite of their deceptively
simple definition, these polynomials possess some truly remarkable
properties. For instance, as shown in \cite{Kn, S-int}, the top homogeneous
part of $R_{\lambda}(x)$ is the Macdonald polynomial $P_{\lambda}(x)$ \cite{Mac}
and $R_\lambda(x)$ satisfies the extra vanishing property $R_\lambda(\overline{\mu})=0$ unless 
$\lambda\subseteq\mu$ as Ferrer diagrams. Other key properties of 
$R_{\lambda}(x) $, which were proven by A. Okounkov \cite{Ok},
include the binomial theorem, which gives an explicit expansion of 
$R_{\lambda}(ax) =R_{\lambda}(ax_{1},\ldots
,ax_{n};q,t) $ in terms of the $R_{\mu}(x;q^{-1},t^{-1})$'s
over the field $\mathbb{K}:=\mathbb{Q}(q,t,a),$ and the \emph{duality} or 
\emph{evaluation symmetry}, which involves the evaluation points 
\begin{equation*}
\widetilde{\mu}=\left( q^{-\mu _{n}}\tau _{1},\ldots ,q^{-\mu _{1}}\tau
_{n}\right),\qquad \mu\in\mathcal{P}_n
\end{equation*}
and takes the form
\begin{equation*}
\frac{R_{\lambda}(a\widetilde{\mu}) }{R_{\lambda }(a\tau)}=\frac{R_{\mu}(a\widetilde{\lambda})}{R_\mu(a\tau)}.
\end{equation*}

The interpolation polynomials have natural non-symmetric analogs $G_\alpha(x)=
G_{\alpha}(x;q,t)$, which were also defined in \cite{Kn, S-int}. These
are indexed by the set of compositions with at most $n$ parts, $\mathcal{C}_n:=
\bigl(\mathbb{Z}_{\geq 0}\bigr)^n$. For a composition 
$\beta\in\mathcal{C}_n$ we define
\begin{equation*}
\overline{\beta}:=w_{\beta}(\overline{\beta _{+}})
\end{equation*}
where $w_{\beta}$ is the shortest permutation such that $\beta
_{+}=w_{\beta }^{-1}(\beta) $ is a partition. Then $G_{\alpha}(x)$
is, up to normalization, characterized as the unique polynomial of degree at most $\left\vert
\alpha \right\vert :=\alpha _{1}+\cdots +\alpha _{n}$ satisfying the vanishing
conditions
\begin{equation*}
G_{\alpha}(\overline{\beta}) =0\text{ for } \beta\in\mathcal{C}_n \text{ such that } \left\vert \beta
\right\vert \leq \left\vert \alpha \right\vert ,\text{ }\beta \neq \alpha.
\end{equation*}
The normalization is fixed by requiring that the coefficient of $x^\alpha:=x_{1}^{\alpha
_{1}}\cdots x_{n}^{\alpha _{n}}$ in the monomial expansion of $G_\alpha(x)$ is $1$.

Many properties of the symmetric interpolation polynomials $R_{\lambda}(x)$
admit non-symmetric counterparts for the $G_{\alpha}(x)$. For instance, the top
homogeneous part of $G_{\alpha}(x) $ is the non-symmetric
Macdonald polynomial $E_{\alpha}(x)$ and $G_\alpha(x)$ satisfies an extra vanishing property \cite{Kn}. An analog of
the binomial theorem, proved by one of us in \cite[Thm. 1.1]{S-bin}, gives
an explicit expansion of $G_{\alpha}(ax;q,t) $ in terms of a
second family of interpolation polynomials $G_{\alpha}^{\prime}(x)=G_\alpha^\prime(x;q,t)$. 
These latter polynomials are characterized by having the same top
homogeneous part as $G_{\alpha}(x) $, namely the non-symmetric
polynomial $E_{\alpha}(x)$, and the following vanishing
conditions at the evaluation points $\widetilde{\beta}:=\overline{(-w_0\beta)}$, with $w_0$ the longest
element of the symmetric group $S_n$:
\begin{equation*}
G_{\alpha}^\prime(\widetilde{\beta}) =0\text{ for }\left\vert \beta
\right\vert <\left\vert \alpha \right\vert.
\end{equation*}

The first result of the present paper is a Demazure-type formula for the
primed interpolation polynomials $G_{\alpha }^{\prime}(x) $ in
terms of $G_{\alpha}(x)$, which involves the symmetric group action on the algebra of polynomials in $n$ variables over $\mathbb{F}$ by permuting the variables, as well as the
associated Hecke algebra action 
in terms of Demazure-Lusztig operators $H_w$ ($w\in S_n$) as described in the
next section.
\vspace{.3cm}

\noindent
{\bf Theorem A.}
{\it
Write
$\textup{I}(\alpha):=\#\left\{i<j\mid \alpha_i\geq\alpha_j\right\}$.
Then we have 
\begin{equation*}
G_{\alpha }^{\prime }(t^{n-1}x;q^{-1},t^{-1})=
t^{(n-1)|\alpha |-
\textup{I}(\alpha)}
w_{0}H_{w_{0}}G_{\alpha }(x;q,t).
\end{equation*}}

This is restated and proved in Theorem \ref{THMprime} below.

The second result is the following duality theorem for $G_{\alpha}(x)$, which
is the non-symmetric analog of Okounkov's duality result.
\vspace{.3cm}

\noindent
{\bf Theorem B.}
{\it 
For all compositions $\alpha,\beta\in\mathcal{C}_n$ we have 
\begin{equation*}
\frac{G_{\alpha}(a\widetilde{\beta})}{G_{\alpha}(a\tau)}=\frac{G_{\beta}(a\widetilde{\alpha})}
{G_{\beta}(a\tau)}.
\end{equation*}}

This is a special case of Theorem \ref{THMsymmetry} below.

We now recall the interpolation $O$-polynomials introduced in \cite[Thm. 1.1]
{S-bin}. Write $x^{-1}$ for $\left( x_{1}^{-1},\ldots
,x_{n}^{-1}\right) $. Then it was shown in \cite[Thm. 1.1]{S-bin} that there
exists a unique polynomial $O_{\alpha}(x)=O_\alpha(x;q,t;a)$ of degree at most
$\left\vert \alpha \right\vert $ with coefficients in the field $\mathbb{K}$ such that
\begin{equation*}
O_{\alpha}(\overline{\beta}^{-1})=\frac{G_{\beta}(a\widetilde{\alpha})}
{G_{\beta}(a\tau)}\text{ for all }\beta .
\end{equation*}
Our third result is a simple expression for the $O$-polynomials in terms of
the interpolation polynomials $G_{\alpha}(x)$.
\vspace{.3cm}

\noindent
{\bf Theorem C.}
{\it For all compositions $\alpha\in\mathcal{C}_n$ 
we have%
\begin{equation*}
O_{\alpha}\left( x\right) =\frac{G_{\alpha}(t^{1-n}aw_{0}x)}{%
G_{\alpha}(a\tau)}.
\end{equation*}}

This is deduced in Proposition \ref{propO} below as a direct consequence of
non-symmetric duality. We also obtain new proofs of Okounkov's \cite{Ok} duality
theorem, as well as the dual binomial theorem of A. Lascoux, E. Rains and O. Warnaar \cite{LRW},
which gives an expansion of the primed-interpolation polynomials $G_{\alpha
}^{\prime }(x) $ in terms of the $G_{\beta}(ax)$'s.\\

\noindent
{\it Acknowledgements:}
We thank Eric Rains for sharing with us his unpublished results with 
Alain Lascoux and Ole Warnaar on a one-parameter rational extension of the non-symmetric interpolation Macdonald polynomials. It leads to a different proof of the duality 
of the non-symmetric interpolation Macdonald polynomials (Theorem B). We thank an anonymous
referee for detailed comments.

The research of S. Sahi was partially supported by a Simons Foundation grant
(509766).

\section{Demazure-Lusztig operators and the primed interpolation polynomials}\label{s2}
We use the notations from \cite{S-bin}. The correspondence with the notations from the other important references \cite{Kn}, \cite{S-int} and \cite{Ok} is listed in \cite[\S 2]{S-bin} (directly after Lemma 2.8).

Let $S_n$ be the symmetric group in $n$ letters and $s_i\in S_n$ the permutation that
swaps $i$ and $i+1$. The $s_i$ ($1\leq i<n$) are Coxeter generators for $S_n$. Let $\ell: S_n\rightarrow\mathbb{Z}_{\geq 0}$ be the associated length function. Let $S_n$ act
on $\mathbb{Z}^n$ and $\mathbb{K}^n$ by $s_iv:=(\cdots,v_{i-1},v_{i+1},v_i,v_{i+2},\ldots)$ for
$v=(v_1,\ldots,v_n)$. Write $w_0\in S_n$ for the longest element, given explicitly by $i\rightarrow n+1-i$ for $i=1,\ldots,n$.

For $v=(v_1,\ldots,v_n)\in\mathbb{Z}^n$ define $\overline{v}=(\overline{v}_1,\ldots,\overline{v}_n)\in\mathbb{F}^n$ by
$\overline{v}_i:=q^{v_i}t^{-k_i(v)}$ with 
\[
k_i(v):=\#\{k<i \,\, | \,\, v_k\geq v_i\}+\#\{k>i \,\, | \,\, v_k>v_i\}.
\]
If $v\in\mathbb{Z}^n$ has non-increasing entries $v_1\geq v_2\geq \cdots\geq v_n$, 
then $\overline{v}=(q^{v_1}\tau_1,
\ldots,q^{v_n}\tau_n)$.
For arbitrary $v\in\mathbb{Z}^n$ we have $\overline{v}=w_v(\overline{v_+})$ with $w_v\in S_n$ the shortest permutation such that
$v_+:=w_v^{-1}(v)$ has nonincreasing entries, see \cite[\S 2]{Kn}. We write $\widetilde{v}:=\overline{-w_0v}$ for $v\in\mathbb{Z}^n$. 

Note that $\overline{\alpha}_n=t^{1-n}$ if 
$\alpha\in\mathcal{C}_n$ with
$\alpha_n=0$.

For a field $F$ we write $F[x]:=F[x_1,\ldots,x_n]$, $F[x^{\pm 1}]:=
F[x_1^{\pm 1},\ldots,x_n^{\pm 1}]$ and $F(x)$ for the quotient field of $F[x]$.
The symmetric group acts by algebra automorphisms on $\mathbb{F}[x]$ and $\mathbb{F}(x)$, with the action of $s_i$ by 
interchanging $x_i$ and $x_{i+1}$ for $1\leq i<n$.
Consider the $\mathbb{F}$-linear operators
\[
H_i=ts_i-\frac{(1-t)x_i}{x_i-x_{i+1}}(1-s_i)=t+\frac{x_i-tx_{i+1}}{x_i-x_{i+1}}(s_i-1)
\]
on $\mathbb{F}(x)$
($1\leq i<n$) called Demazure-Lusztig operators, and the automorphism $\Delta$ of $\mathbb{F}(x)$ defined by
\[
\Delta f(x_1,\dots,x_n)=f(q^{-1}x_n,x_1,\ldots,x_{n-1}).
\]
Note that $H_i$ ($1\leq i<n$) and $\Delta$ preserve $\mathbb{F}[x^{\pm 1}]$ and $\mathbb{F}[x]$.
Cherednik \cite{ChPap,Ch} showed that the operators $H_i$ ($1\leq i<n$) and $\Delta$ satisfy the defining relations of the type A extended affine Hecke algebra,
\begin{equation*}
\begin{split}
(H_i-t)(H_i+1)&=0,\\
H_iH_j&=H_jH_i,\qquad |i-j|>1,\\
H_iH_{i+1}H_i&=H_{i+1}H_iH_{i+1},\\
\Delta H_{i+1}&=H_i\Delta,\\
\Delta^2 H_{1}&=H_{n-1}\Delta^2
\end{split}
\end{equation*}
for all the indices such that both sides of the equation make sense
(see also \cite[\S 3]{Kn}). For $w\in S_n$ we write $H_w:=H_{i_1}H_{i_2}\cdots H_{i_\ell}$
with $w=s_{i_1}s_{i_2}\cdots s_{i_\ell}$ a reduced expression for $w\in S_n$. It is well defined
because of the braid relations for the $H_i$'s. Write $\overline{H}_i:=H_i+1-t=tH_i^{-1}$ and set
\begin{equation}\label{Cherednik}
\xi_i:=t^{1-n}\overline{H}_{i-1}\cdots\overline{H}_1\Delta^{-1}H_{n-1}\cdots H_i,\qquad
1\leq i\leq n.
\end{equation}
The operators $\xi_i$'s are pairwise commuting invertible operators, with inverses
\[
\xi_i^{-1}=\overline{H}_i\cdots \overline{H}_{n-1}\Delta H_1\cdots H_{i-1}.
\]
The $\xi_i^{-1}$ ($1\leq i\leq n$) are the Cherednik operators \cite{Ch,Kn}.

The monic non-symmetric Macdonald polynomial $E_\alpha\in\mathbb{F}[x]$ of degree $\alpha\in\mathcal{C}_n$ is the unique polynomial satisfying
\[
\xi_i^{-1}E_\alpha=\overline{\alpha}_iE_\alpha,\qquad i=1,\ldots,n
\]
and normalized such that the coefficient of $x^\alpha$ in $E_\alpha$ is $1$.

Let $\iota$ be the field automorphism of $\mathbb{K}$ inverting $q$, $t$ and $a$. It restricts to
a field automorphism of  $\mathbb{F}$, inverting $q$ and $t$. We extend $\iota$ to a $\mathbb{Q}$-algebra automorphism of $\mathbb{K}[x]$ and $\mathbb{F}[x]$ by letting $\iota$ act on the coefficients of the polynomial. Write 
\[
G_\alpha^\circ:=\iota\bigl(G_\alpha\bigr), \qquad E_\alpha^\circ:=\iota\bigl(E_\alpha\bigr)
\]
for $\alpha\in\mathcal{C}_n$. 
Note that $\overline{v}^{-1}=(\iota(\overline{v}_1),\ldots,\iota(\overline{v}_n))$. 

Put $H_i^\circ$, $H_w^\circ$, $\overline{H}_i^\circ$, $\Delta^\circ$ and $\xi_i^\circ$ for the operators 
$H_i$, $H_w$, $\overline{H}_i$, $\Delta$ and $\xi_i$ with $q,t$ replaced by their inverses. For instance,
\begin{equation*}
\begin{split}
&H_i^\circ=t^{-1}s_i-\frac{(1-t^{-1})x_i}{x_i-x_{i+1}}(1-s_i),\\
&\Delta^\circ f(x_1,\ldots,x_n)=f(qx_n,x_1,\ldots,x_{n-1}).
\end{split}
\end{equation*}
We then have
$\xi_i^\circ E_\alpha^\circ=\overline{\alpha}_iE_\alpha^\circ$ for $i=1,\ldots,n$, which characterizes
$E_\alpha^\circ$ up to a scalar factor.

\begin{thm}\label{THMprime}
For $\alpha\in\mathcal{C}_n$ we have
\begin{equation}\label{identity}
G_\alpha^\prime(x)=t^{(1-n)|\alpha|+\textup{I}(\alpha)}w_0H_{w_0}^\circ G_\alpha^\circ(t^{n-1}x)
\end{equation}
with $\textup{I}(\alpha):=\#\{i<j\, | \, \alpha_i\geq \alpha_j\}$.
\end{thm}
\noindent
{\bf Remark.}
Formally set $t=q^r$, replace $x$ by $1+(q-1)x$, divide both sides of \eqref{identity} by $(q-1)^{|\alpha|}$ and take the limit $q\rightarrow 1$. Then
\begin{equation}\label{Jackversion}
G_\alpha^\prime(x;r)=(-1)^{|\alpha|}\sigma(w_0)w_0G_\alpha(-x-(n-1)r;r)
\end{equation}
for the non-symmetric interpolation Jack polynomial $G_\alpha(\cdot\,;r)$ and its primed version (see \cite{S-bin}). Here $\sigma$ denotes the action of the symmetric group with $\sigma(s_i)$ the rational degeneration of the Demazure-Lusztig operators $H_i$, given explicitly by
\[
\sigma(s_i)=s_i+\frac{r}{x_i-x_{i+1}}(1-s_i),
\]
see \cite[\S 1]{S-bin}. To establish the formal limit \eqref{Jackversion} one uses that $\sigma(w_0)w_0=w_0\sigma^\circ(w_0)$ with $\sigma^\circ$ the action of the symmetric group defined in terms of the rational degeneration
\[
\sigma^\circ(s_i)=s_i-\frac{r}{x_i-x_{i+1}}(1-s_i)
\]
of $H_i^\circ$. Formula \eqref{Jackversion} was obtained before in \cite[Thm. 1.10]{S-bin}.
\begin{proof}
We show that the right hand side of \eqref{identity} satisfies the defining properties of $G_\alpha^\prime$. For the vanishing property, note that
\begin{equation}\label{tildeversusnormal}
t^{n-1}w_0\widetilde{\beta}=\overline{\beta}^{-1}
\end{equation}
 (this is the $q$-analog of \cite[Lem. 6.1(2)]{S-bin}), hence
\[\bigl(w_0H_{w_0}^\circ G_\alpha^\circ(t^{n-1}x)\bigr)|_{x=\widetilde{\beta}}=
\bigl(H_{w_0}^\circ G_\alpha^\circ(x)\bigr)|_{x=\overline{\beta}^{-1}}.
\]
This expression is zero for $|\beta|<|\alpha|$ since it is a linear combination of the evaluated interpolation polynomials
$G_\alpha^\circ(\overline{w\beta}^{-1})$ ($w\in S_n$) 
by \cite[Lem. 2.1(2)]{S-bin}. 

It remains to show that the top homogeneous
terms of both sides of \eqref{identity} are the same, i.e. that
\begin{equation}\label{step}
E_\alpha=t^{\textup{I}(\alpha)}w_0H_{w_0}^\circ E_\alpha^\circ.
\end{equation}

Note that $\Psi:=w_0H_{w_0}^\circ$ satisfies the intertwining properties
\begin{equation}\label{int}
\begin{split}
H_i\Psi&=t\Psi \overline{H}_i^\circ,\\
\Delta\Psi&=t^{n-1}\Psi \overline{H}_{n-1}^\circ\cdots\overline{H}_1^\circ(\Delta^\circ)^{-1}H_{n-1}^\circ\cdots H_1^\circ
\end{split}
\end{equation}
for $1\leq i<n$ (use e.g. \cite[Prop. 3.2.2]{Ch}). It follows that $\xi_i^{-1}\Psi=\Psi\xi_i^\circ$ for $i=1,\ldots,n$.
Therefore
\[
E_\alpha(x)=c_\alpha\Psi E_\alpha^\circ(x)
\]
for some constant $c_\alpha\in\mathbb{F}$. But the coefficient of $x^\alpha$ in
$\Psi x^\alpha$ is $t^{-\textup{I}(\alpha)}$,
hence $c_\alpha=t^{\textup{I}(\alpha)}$.
\end{proof}
\noindent

Consider the Demazure operators $H_i$ and the Cherednik operators $\xi_j^{-1}$ as operators on 
the space $\mathbb{F}[x^{\pm 1}]$ of Laurent polynomials.
For an integral vector $u\in\mathbb{Z}^n$, let $E_u\in\mathbb{F}[x^{\pm 1}]$ be the common eigenfunction of the Cherednik operators $\xi_j^{-1}$ with eigenvalues $\overline{u}_j$ ($1\leq j\leq n$), normalized such that the coefficient of $x^u:=x_1^{u_1}\cdots x_n^{u_n}$ in $E_u$ is $1$. For 
$u=\alpha\in\mathcal{C}_n$ this definition reproduces the non-symmetric Macdonald polynomial $E_\alpha\in\mathbb{F}[x]$ as defined before. Note that 
\[
E_{u+(1^n)}=x_1\cdots x_nE_u(x).
\]
It is now easy to check that formula \eqref{step} is valid with $\alpha$ replaced by an arbitrary integral vector $u$,
\begin{equation}\label{stepprime}
E_u=t^{\textup{I}(u)}w_0H_{w_0}^\circ E_u^\circ
\end{equation}
with $E_u^\circ:=\iota(E_u)$.
Furthermore, one can show in the same vein as the proof of \eqref{step} that
\[
w_0E_{-w_0u}(x^{-1})=E_u(x)
\]
for an integral vector $u$, where $p(x^{-1})$ stands for inverting all the parameters $x_1,\ldots,x_n$
in the Laurent polynomial $p(x)\in\mathbb{F}[x^{\pm 1}]$.
Combining this equality with  \eqref{stepprime} yields
\[
E_{-w_0u}(x^{-1})=t^{\textup{I}(u)}H_{w_0}^\circ E_u^\circ(x),
\]
which is a  special case of a known identity for non-symmetric Macdonald polynomials
(see \cite[Prop. 3.3.3]{Ch}).

\section{Evaluation formulas}
In \cite[Thm. 1.1]{S-bin} the following combinatorial evaluation formula 
\begin{equation}\label{PrincipalEvaluation}
G_\alpha(a\tau)=\prod_{s\in\alpha}\Bigl(\frac{t^{1-n}-q^{a^\prime(s)+1}t^{1-l^\prime(s)}}
{1-q^{a(s)+1}t^{l(s)+1}}\Bigr)\prod_{s\in\alpha}(at^{l^\prime(s)}-q^{a^\prime(s)})
\end{equation}
was obtained, with $a(s)$, $l(s)$, $a^\prime(s)$ and $l^\prime(s)$ the arm, leg, coarm and coleg of 
$s=(i,j)\in\alpha$, defined by
\begin{equation*}
\begin{split}
a(s)&:=\alpha_i-j,\qquad l(s):=\#\{k>i \,\, | \,\, j\leq\alpha_k\leq\alpha_i\}+
\#\{k<i\,\, | \,\, j\leq\alpha_k+1\leq\alpha_i\},\\
a^\prime(s)&:=j-1,\qquad l^\prime(s):=\#\{k>i\,\, | \,\, \alpha_k>\alpha_i\}+\#\{k<i\,\, | \,\, \alpha_k
\geq\alpha_i\}.
\end{split}
\end{equation*}
By \eqref{PrincipalEvaluation} we have
\[
E_\alpha(\tau)=\lim_{a\rightarrow\infty}a^{-|\alpha|}G_\alpha(a\tau)=\prod_{s\in \alpha}
\Bigl(\frac{t^{1-n+l^\prime(s)}-q^{a^\prime(s)+1}t}{1-q^{a(s)+1}t^{l(s)+1}}\Bigr),
\] 
which is the well known evaluation formula \cite{ChPap,Ch} for the non-symmetric Macdonald polynomials. Note that for $\alpha\in\mathcal{C}_n$, 
\[
\ell(w_0)-I(\alpha)=\#\{i<j \, | \, \alpha_i<\alpha_j\}.
\]
\begin{lem}\label{CORevalprime}
For $\alpha\in\mathcal{C}_n$ we have
\[
G_\alpha^\prime(a\tau)=t^{(1-n)|\alpha|+\textup{I}(\alpha)-\ell(w_0)}G_\alpha^\circ(a\tau^{-1}).
\]
\end{lem}
\begin{proof}
Since $t^{n-1}w_0\tau=\tau^{-1}=\overline{0}^{-1}$ we have by Theorem \ref{THMprime},
\begin{equation*}
\begin{split}
G_\alpha^\prime(a\tau)&=t^{(1-n)|\alpha|+\textup{I}(\alpha)}\bigl(H_{w_0}^{\circ}G_\alpha^\circ\bigr)(a\overline{0}^{-1})\\
&=t^{(1-n)|\alpha|+\textup{I}(\alpha)-\ell(w_0)}G_\alpha^\circ(a\overline{0}^{-1}),
\end{split}
\end{equation*}
where we have used \cite[Lem. 2.1(2)]{S-bin} for the second equality.
\end{proof}

We now derive a relation between the evaluation formulas for $G_\alpha(x)$ and $G_\alpha^\circ(x)$. To formulate this we write, following
\cite{LRW},
\[
n(\alpha):=\sum_{s\in\alpha}l(s),\qquad n^\prime(\alpha):=\sum_{s\in\alpha}a(s).
\]
Note that $n^\prime(\alpha)=\sum_{i=1}^n{{\alpha_i}\choose{2}}$,
hence it only depends on the $S_n$-orbit of $\alpha$, while
\begin{equation}\label{orbitrel}
n(\alpha)=n(\alpha^+)+\ell(w_0)-I(\alpha).
\end{equation}
The following lemma is a non-symmetric version of the first displayed formula on
\cite[Page 537]{Ok}.
\begin{lem}\label{LEMevalinverse}
For $\alpha\in\mathcal{C}_n$ we have
\[
G_\alpha(a\tau)=(-a)^{|\alpha|}t^{(1-n)|\alpha|-n(\alpha)}q^{n^\prime(\alpha)}G_\alpha^\circ(a^{-1}\tau^{-1}).
\]
\end{lem}
\begin{proof}
This follows from
the explicit evaluation formula \eqref{PrincipalEvaluation} for the non-symmetric interpolation Macdonald polynomial $G_\alpha$.
\end{proof}

Following \cite[(3.9)]{LRW} we define $\tau_\alpha\in\mathbb{F}$ ($\alpha\in\mathcal{C}_n$) by
\begin{equation}\label{taualpha}
\tau_\alpha:=(-1)^{|\alpha|}q^{n^\prime(\alpha)}t^{-n(\alpha^+)}.
\end{equation}
It only depends on the $S_n$-orbit of $\alpha$. 
\begin{cor}\label{CORrelev}
For $\alpha\in\mathcal{C}_n$ we have
\[
G_\alpha^\prime(a^{-1}\tau)=\tau_\alpha^{-1}a^{-|\alpha|}G_\alpha(a\tau).
\]
\end{cor}
\begin{proof}
Use Lemma \ref{CORevalprime}, Lemma \ref{LEMevalinverse} and \eqref{orbitrel}.
\end{proof}

\section{Normalized interpolation Macdonald polynomials}\label{SECTIONnormalised}

We need the basic representation of the (double) affine Hecke algebra on the space of $\mathbb{K}$-valued functions on $\mathbb{Z}^n$, which is constructed as follows. 

For $v\in\mathbb{Z}^n$ and $y\in\mathbb{K}^n$ write 
 $v^\natural:=(v_2,\ldots,v_n,v_1+1)$
 and $y^\natural:=(y_2,\ldots,y_n,qy_1)$. 
Denote the inverse of ${}^\natural$ by ${}^\sharp$, so $v^\sharp=(v_n-1,v_1,\ldots,v_{n-1})$ and $y^\sharp=(y_n/q,y_1,\ldots,y_{n-1})$.
We have the following lemma (cf. \cite{Kn,S-bin,S-int}). 
\begin{lem}\label{LEMtech}
Let $v\in\mathbb{Z}^n$ and $1\leq i<n$. Then we have
\begin{enumerate}
\item[{\bf 1.}] $s_i(\overline{v})=\overline{s_iv}$ if $v_i\not=v_{i+1}$.
\item[{\bf 2.}] $\overline{v}_i=t\overline{v}_{i+1}$ if $v_i=v_{i+1}$.
\item[{\bf 3.}] $\overline{v}^\natural=\overline{v^\natural}$.
\end{enumerate}
\end{lem} 
Let $\mathbb{H}$ be the double affine Hecke algebra over $\mathbb{K}$. It is isomorphic to the subalgebra of $\textup{End}(\mathbb{K}[x^{\pm 1}])$ generated
by the operators $H_i$ ($1\leq i<n$), $\Delta^{\pm 1}$, 
and the multiplication operators $x_j^{\pm 1}$ ($1\leq j\leq n$). 

For a unital $\mathbb{K}$-algebra $A$ we write $\mathcal{F}_A$ for the space of $A$-valued functions $f: \mathbb{Z}^n\rightarrow A$ on $\mathbb{Z}^n$.
\begin{cor}\label{CORdiscrete}
Let $A$ be a unital $\mathbb{K}$-algebra. Consider the $A$-linear operators $\widehat{H}_i$ ($1\leq i<n$), $\widehat{\Delta}$
and $\widehat{x}_j$ ($1\leq j\leq n$) on $\mathcal{F}_A$ defined by 
\begin{equation}\label{hataction}
\begin{split}
(\widehat{H}_if)(v)&:=tf(v)+\frac{\overline{v}_i-t\overline{v}_{i+1}}{\overline{v}_i-\overline{v}_{i+1}}(f(s_iv)-f(v)),\\
(\widehat{\Delta} f)(v)&:=f(v^\sharp),\qquad (\widehat{\Delta}^{-1}f)(v):=f(v^\natural),\\
 (\widehat{x}_jf)(v)&:=a\overline{v}_jf(v)
\end{split}
\end{equation}
for $f\in\mathcal{F}_A$ and $v\in\mathbb{Z}^n$. Then $H_i\mapsto \widehat{H}_i$ ($1\leq i<n$),
$\Delta\mapsto \widehat{\Delta}$ and $x_j\mapsto \widehat{x}_j$ ($1\leq j\leq n$)
defines a representation $\mathbb{H}\rightarrow\textup{End}_{A}(\mathcal{F}_A)$, $X\mapsto 
\widehat{X}$ ($X\in\mathbb{H}$)
of the double affine Hecke algebra $\mathbb{H}$ on $\mathcal{F}_A$.
\end{cor}
\begin{proof}
Let $\mathcal{O}\subset \mathbb{K}^n$ be the smallest $S_n$-invariant and $\natural$-invariant subset which contains $\{a\overline{v}\,\, | \,\, v\in\mathbb{Z}^n\}$. Note that  
$\mathcal{O}$ is contained in $\{y\in\mathbb{K}^n \,\, | \,\, y_i\not=y_j \hbox{ if } i\not=j\}$. 
The Demazure-Lusztig operators $H_i$ ($1\leq i<n$), $\Delta^{\pm 1}$ and the coordinate multiplication operators $x_j$ ($1\leq j\leq n$) act $A$-linearly on the space $F_A^{\mathcal{O}}$ of $A$-valued functions on 
$\mathcal{O}$, and hence turns $F_A^{\mathcal{O}}$ 
into a 
$\mathbb{H}$-module. Define the surjective $A$-linear map
\[
\textup{pr}: 
F^{\mathcal{O}}_A\rightarrow\mathcal{F}_A
\]
by $\textup{pr}(g)(v):=g(a\overline{v})$ ($v\in\mathbb{Z}^n$). 

We claim that $\textup{Ker}(\textup{pr})$ is a $\mathbb{H}$-submodule of $F^{\mathcal{O}}_A$.
Clearly $\textup{Ker}(\textup{pr})$ is $x_j$-invariant for $j=1,\ldots,n$. 
Let $g\in\textup{Ker}(\textup{pr})$. 
Part 3 of Lemma \ref{LEMtech} implies that $\Delta g\in\textup{Ker}(\textup{pr})$. To show that $H_ig\in\textup{Ker}(\textup{pr})$
we consider two cases. If $v_i\not=v_{i+1}$ then $s_i\overline{v}=\overline{s_iv}$
by part 1 of Lemma \ref{LEMtech}. Hence 
\[
(H_ig)(a\overline{v})=tg(a\overline{v})+\frac{\overline{v}_i-t\overline{v}_{i+1}}{\overline{v}_i-\overline{v}_{i+1}}(g(a\overline{s_iv})-g(a\overline{v}))=0.
\]
If $v_i=v_{i+1}$ then $\overline{v}_i=t\overline{v}_{i+1}$ by part 2 of Lemma \ref{LEMtech}. Hence 
\[
(H_ig)(\overline{v})=tg(a\overline{v})+\frac{\overline{v}_i-t\overline{v}_{i+1}}{\overline{v}_i-\overline{v}_{i+1}}(g(as_i\overline{v})-g(a\overline{v}))=
tg(a\overline{v})=0.
\]

Hence $\mathcal{F}_A$ inherits the $\mathbb{H}$-module structure of $F^{\mathcal{O}}_A/\textup{Ker}(\textup{pr})$. It is a straightforward
computation, using Lemma \ref{LEMtech} again, to show that the resulting action of $H_i$ ($1\leq i<n$), $\Delta$ and $x_j$ ($1\leq j\leq n$) on $\mathcal{F}_A$ is by the operators $\widehat{H}_i$ ($1\leq i<n$), $\widehat{\Delta}$ and 
$\widehat{x}_j$ ($1\leq j\leq n$).
\end{proof}
\begin{rema}\label{HatRelation}
With the notations from (the proof of) Corollary \ref{CORdiscrete}, let $\widetilde{g}\in 
F_A^{\mathcal{O}}$ and set $g:=\textup{pr}(\widetilde{g})\in \mathcal{F}_A$. In other words,
$g(v):=\widetilde{g}(a\overline{v})$ for all $v\in \mathbb{Z}^n$. Then 
\[
\bigl(\widehat{X}g\bigr)(v)=(X\widetilde{g})(a\overline{v}),\qquad v\in\mathbb{Z}^n
\]
for 
$X=H_i,\Delta^{\pm 1}, x_j$.
\end{rema}
\begin{rema}\label{F+}
Let $\mathcal{F}^+_A$ be the space of $A$-valued functions on $\mathcal{C}_n$.
We sometimes will consider $\widehat{H}_i$ ($1\leq i<n$), $\widehat{\Delta}^{-1}$ and $\widehat{x}_j$ ($1\leq j\leq n$), defined by the formulas \eqref{hataction}, as linear operators on $\mathcal{F}^+_A$.
\end{rema}

\begin{defi}
We call 
\begin{equation}\label{Kalpha}
K_\alpha(x;q,t;a):=\frac{G_\alpha(x;q,t)}{G_\alpha(a\tau;q,t)}\in\mathbb{K}[x]
\end{equation}
the {\it normalized non-symmetric interpolation Macdonald polynomial} of degree $\alpha$. 
\end{defi}
We frequently use the shorthand notation $K_\alpha(x):=K_\alpha(x;q,t;a)$. We will see in a moment that formulas for non-symmetric interpolation Macdonald polynomials take the nicest form in this particular normalization. 

Note that $a$ cannot be specialized to $1$ in \eqref{Kalpha} since $G_\alpha(\tau)=G_\alpha(\overline{0})=0$ if $0\not=\alpha\in\mathcal{C}_n$.
Note furthermore that 
\begin{equation}\label{limitK}
\lim_{a\rightarrow \infty}K_\alpha(ax)=
\frac{E_\alpha(x)}{E_\alpha(\tau)}
\end{equation}
since $\lim_{a\rightarrow\infty}a^{-|\alpha|}G_\alpha(ax)=E_\alpha(x)$.
 
Recall from \cite{Kn} the operator $\Phi=(x_n-t^{1-n})\Delta\in\mathbb{H}$ and the inhomogeneous Cherednik operators
\[
\Xi_j=\frac{1}{x_j}+\frac{1}{x_j}H_j\cdots H_{n-1}\Phi H_1\cdots H_{j-1}\in\mathbb{H},\qquad 1\leq j\leq n.
\]
The operators $H_i$, $\Xi_j$ and $\Phi$ preserve $\mathbb{K}[x]$ (see \cite{Kn}), hence they give rise to $\mathbb{K}$-linear operators on $\mathcal{F}^+_{\mathbb{K}[x]}$
(e.g., $(H_if)(\alpha):=H_i(f(\alpha))$ for $\alpha\in\mathcal{C}_n$).
Note that the operators $H_i,\Xi_j$ and $\Phi$ on $\mathcal{F}_{\mathbb{K}[x]}^+$ commute with the hat-operators
$\widehat{H}_i$, $\widehat{x}_j$ and $\widehat{\Delta}^{-1}$ on $\mathcal{F}_{\mathbb{K}[x]}^+$
(cf. Remark \ref{F+}). The same remarks hold true for the space 
$\mathcal{F}_{\mathbb{K}(x)}$ of $\mathbb{K}(x)$-valued functions on $\mathbb{Z}^n$ (in fact, in this case the hat-operators define a $\mathbb{H}$-action on $\mathcal{F}_{\mathbb{K}(x)}$).

Let $K\in\mathcal{F}^+_{\mathbb{K}[x]}$ be the map $\alpha\mapsto K_\alpha(\cdot)$ ($\alpha\in\mathcal{C}_n$).

\begin{lem}\label{LEMactiontransfer}
For $1\leq i<n$ and $1\leq j\leq n$ we have in $\mathcal{F}_{\mathbb{K}[x]}^+$,
\begin{enumerate}
\item[{\bf 1.}]
$H_iK=\widehat{H}_iK$.
\item[{\bf 2.}] $\Xi_jK=a\widehat{x}_j^{-1}K$.
\item[{\bf 3.}] $\Phi K=t^{1-n}(a^2\widehat{x}_1^{-1}-1)\widehat{\Delta}^{-1}K$.
\end{enumerate}
\end{lem}
\begin{proof}
{\bf 1.} To derive the formula we need to expand $H_iK_\alpha$ as a linear combination of the $K_\beta$'s. As a first step we expand $H_iG_\alpha$ as linear combination of the $G_\beta$'s.

If $\alpha\in\mathcal{C}_n$ satisfies $\alpha_i<\alpha_{i+1}$ then 
\[
H_iG_\alpha(x)=\frac{(t-1)\overline{\alpha}_i}{\overline{\alpha}_i-\overline{\alpha}_{i+1}}G_\alpha(x)+G_{s_i\alpha}(x)
\]
by \cite[Lem. 2.2]{S-bin}.
Using part 1 of Lemma \ref{LEMtech} and the fact that $H_i$ satisfies the quadratic relation $(H_i-t)(H_i+1)=0$, it follows that
\[
H_iG_\alpha(x)=\frac{(t-1)\overline{\alpha}_i}{\overline{\alpha}_i-\overline{\alpha}_{i+1}}G_\alpha(x)+
\frac{t(\overline{\alpha}_{i+1}-t\overline{\alpha}_i)(\overline{\alpha}_{i+1}-t^{-1}\overline{\alpha}_i)}{(\overline{\alpha}_{i+1}-\overline{\alpha}_i)^2}
G_{s_i\alpha}(x)
\]
if $\alpha\in\mathcal{C}_n$ satisfies $\alpha_i>\alpha_{i+1}$. 
Finally, $H_iG_\alpha(x)=tG_\alpha(x)$ if $\alpha\in\mathcal{C}_n$ satisfies $\alpha_i=\alpha_{i+1}$ by \cite[Cor. 3.4]{Kn}. 

An explicit expansion of $H_iK_\alpha$ as linear combination of the $K_\beta$'s can now be obtained using the formula
\begin{equation*}
G_\alpha(a\tau)=\frac{\overline{\alpha}_{i+1}-t\overline{\alpha}_i}{\overline{\alpha}_{i+1}-\overline{\alpha}_i}G_{s_i\alpha}(a\tau)
\end{equation*}
for $\alpha\in\mathcal{C}_n$ satisfying $\alpha_i>\alpha_{i+1}$, cf.\ the proof of \cite[Lem 3.1]{S-bin}. By a direct computation the resulting expansion formula can be written as
$H_iK=\widehat{H}_iK$.\\
{\bf 2.} See \cite[Thm. 2.6]{Kn}.\\
{\bf 3.} Let $\alpha\in\mathcal{C}_n$. By \cite[Lem. 2.2 (1)]{S-bin},
\[
\Phi G_\alpha(x)=q^{-\alpha_1}G_{\alpha^\natural}(x).
\]
By the evaluation formula \eqref{PrincipalEvaluation} we have
\[
\frac{G_{\alpha^\natural}(a\tau)}{G_{\alpha}(a\tau)}=at^{1-n+k_1(\alpha)}-q^{\alpha_1}t^{1-n}.
\]
Hence 
\[
\Phi K_\alpha(x)=t^{1-n}(a\overline{\alpha}_1^{-1}-1)K_{\alpha^\natural}(x).
\]
\end{proof}
\begin{rema}
Note that
\[
\Phi K_\alpha(x)=(a\widetilde{\alpha}_n-t^{1-n})K_{\alpha^\natural}(x)
\]
for $\alpha\in\mathcal{C}_n$ since $\overline{\alpha}^{-1}=t^{n-1}w_0\widetilde{\alpha}$.
\end{rema}
\section{Interpolation Macdonald polynomials with negative degrees}
In this section we give the natural extension of the interpolation Macdonald polynomials $G_\alpha(x)$ and $K_\alpha(x)$ to $\alpha\in\mathbb{Z}^n$. It will be the unique extension of 
$K\in\mathcal{F}_{\mathbb{K}[x]}^+$ to a map $K\in\mathcal{F}_{\mathbb{K}(x)}$
such that Lemma \ref{LEMactiontransfer} remains valid.

\begin{lem}\label{LEMfullshift}
For $\alpha\in\mathcal{C}_n$ we have
\begin{equation*}
\begin{split}
G_\alpha(x)&=q^{-|\alpha|}\frac{G_{\alpha+(1^n)}(qx)}{\prod_{i=1}^n(qx_i-t^{1-n})},\\
K_\alpha(x)&=\Bigl(\prod_{i=1}^n\frac{(1-a\overline{\alpha}_i^{-1})}{(1-qt^{n-1}x_i)}\Bigr)K_{\alpha+(1^n)}(qx).
\end{split}
\end{equation*}
\end{lem}
\begin{proof}
Note that for $f\in\mathbb{K}[x]$,
\[
\Phi^nf(x)=\Bigl(\prod_{i=1}^n(x_i-t^{1-n})\Bigr)f(q^{-1}x).
\]
The first formula then follows by iteration of \cite[Lem. 2.2(1)]{S-bin} and the second formula from part 3 of Lemma \ref{LEMactiontransfer}.
\end{proof}
For $m\in\mathbb{Z}_{\geq 0}$ we define $A_m(x;v)\in\mathbb{K}(x)$ by 
\begin{equation}\label{Am}
A_m(x;v):=\prod_{i=1}^n\frac{\bigl(q^{1-m}a\overline{v}_i^{-1};q\bigr)_m}{\bigl(qt^{n-1}x_i;q\bigr)_m} \qquad\forall\,v\in\mathbb{Z}^n,
\end{equation}
with $\bigl(y;q\bigr)_m:=\prod_{j=0}^{m-1}(1-q^jy)$ the $q$-shifted factorial.
\begin{defi}\label{DEFnegdegree}
Let $v\in\mathbb{Z}^n$ and write $|v|:=v_1+\cdots+v_n$.  Define
$G_v(x)=G_v(x;q,t)\in\mathbb{F}(x)$ and $K_v(x)=K_v(x;q,t;a)\in\mathbb{K}(x)$ by
\begin{equation*}
\begin{split}
G_v(x)&:=q^{-m|v|-m^2n}\frac{G_{v+(m^n)}(q^mx)}{\prod_{i=1}^nx_i^m\bigl(q^{-m}t^{1-n}x_i^{-1};q\bigr)_{m}},\\
K_v(x)&:=A_m(x;v)
K_{v+(m^n)}(q^mx)
\end{split}
\end{equation*}
where $m$ is a nonnegative integer such that $v+(m^n)\in\mathcal{C}_n$ (note that $G_v$ and $K_v$ are well defined by Lemma \ref{LEMfullshift}).
\end{defi}
\begin{eg}\label{n=1example}
If $n=1$ then for $m\in\mathbb{Z}_{\geq 0}$,
\begin{equation*}
K_{-m}(x)=\frac{\bigl(qa;q\bigr)_m}{\bigl(qx;q\bigr)_m},\qquad
K_m(x)=\Bigl(\frac{x}{a}\Bigr)^m\frac{\bigl(x^{-1};q\bigr)_m}{\bigl(a^{-1};q\bigr)_m}.
\end{equation*}
\end{eg}

\begin{lem}\label{LEMstep2}
For all $v\in\mathbb{Z}^n$,
\[
K_v(x)=\frac{G_v(x)}{G_v(a\tau)}.
\]
\end{lem}
\begin{proof}
Let $v\in\mathbb{Z}^n$. Clearly $G_v(x)$ and $K_v(x)$ only differ by a multiplicative constant, so it suffices to show that
$K_v(a\tau)=1$. Fix $m\in\mathbb{Z}_{\geq 0}$ such that $v+(m^n)\in\mathcal{C}_n$.
Then
\[
K_v(a\tau)=A_m(a\tau;v)K_{v+(m^n)}(q^ma\tau)=
A_m(a\tau;v)\frac{G_{v+(m^n)}(q^ma\tau)}{G_{v+(m^n)}(a\tau)}=1,
\]
where the last formula follows from a direct computation using the evaluation formula
\eqref{PrincipalEvaluation}.
\end{proof}

We extend the map $K: \mathcal{C}_n\rightarrow \mathbb{K}[x]$
to a map 
\[
K:\mathbb{Z}^n\rightarrow \mathbb{K}(x)
\]
by setting $v\mapsto K_v(x)$ for all $v\in\mathbb{Z}^n$.  Lemma \ref{LEMactiontransfer} now extends as follows.
\begin{prop}\label{PROPactionsame}
We have, as identities in $\mathcal{F}_{\mathbb{K}(x)}$,
\begin{enumerate}
\item[{\bf 1.}]
$H_iK=\widehat{H}_iK$.
\item[{\bf 2.}] $\Xi_jK=a\widehat{x}_j^{-1}K$.
\item[{\bf 3.}] $\Phi K=t^{1-n}(a^2\widehat{x}_1^{-1}-1)\widehat{\Delta}^{-1}K$.
\end{enumerate}
\end{prop}
\begin{proof}
Write $A_m\in\mathcal{F}_{\mathbb{K}(x)}$ for the map
$v\mapsto A_m(x;v)$ for $v\in\mathbb{Z}^n$.
Consider the linear operator on $\mathcal{F}_{\mathbb{K}(x)}$ defined by 
$(A_mf)(v):=A_m(x;v)f(v)$ for $v\in \mathbb{Z}^n$ and $f\in\mathcal{F}_{\mathbb{K}(x)}$.
For $1\leq i<n$ we have $[H_i,A_m]=0$  as linear operators on $\mathcal{F}_{\mathbb{K}(x)}$,
since $A_m(x;v)$ is a symmetric rational function in $x_1,\ldots,x_n$.  Furthermore, for $v\in\mathbb{Z}^n$ and $f\in\mathcal{F}_{\mathbb{K}(x)}$,
\begin{equation}\label{formula1}
\bigl((\widehat{H}_i\circ A_m)f\bigr)(v)=\bigl((A_m\circ \widehat{H}_i)f\bigr)(v)\quad
\hbox{ if }\,\,\, v_i\not=v_{i+1}
\end{equation}
by part 2 of Lemma \ref{LEMtech} and the fact that $A_m(x;v)$ is symmetric in $\overline{v}_1,\ldots,\overline{v}_n$.
Fix $v\in\mathbb{Z}^n$ and choose $m\in\mathbb{Z}_{\geq 0}$ such that $v+(m^n)\in
\mathcal{C}_n$. Since 
\[
K_v(x)=A_m(x;v)K_{v+(m^n)}(q^mx)
\]
we obtain from $[H_i,A_m]=0$ and \eqref{formula1} that $(H_iK)(v)=(\widehat{H}_iK)(v)$
if $v_i\not=v_{i+1}$. This also holds true if $v_i=v_{i+1}$ since then $(\widehat{H}_iK)(v)=
tK_v$ and $H_iK_{v+(m^n)}(q^mx)=tK_{v+(m^n)}(q^mx)$. This proves part 1 of the proposition.

Note that $\Phi K_v(x)=t^{1-n}(a\overline{v}_1^{-1}-1)K_{v^\natural}(x)$ for arbitrary $v\in\mathbb{Z}^n$ by Lemma \ref{LEMactiontransfer} and the commutation relation
\begin{equation}\label{formula2}
\Phi\circ A_m=A_m\circ \Phi^{(q^m)},
\end{equation}
where $\Phi^{(q^m)}:=(q^mx_n-t^{1-n})\Delta$. This proves part 3 of the proposition.

Finally we have $\Xi_jK_v(x)=\overline{v}_j^{-1}K_v(x)$ for all $v\in\mathbb{Z}^n$ by $[H_i,A_m]=0$, \eqref{formula2} and Lemma \ref{LEMactiontransfer}. This proves part 2 of the proposition.
\end{proof}

\section{Duality of the non-symmetric interpolation Macdonald polynomials}

Recall
the notation $\widetilde{v}=\overline{-w_0v}$ for $v\in\mathbb{Z}^n$. 
\begin{thm}[Duality]\label{THMsymmetry}
For all $u,v\in\mathbb{Z}^n$ we have
\begin{equation}\label{Fsymmetry}
K_u(a\widetilde{v})=K_v(a\widetilde{u}).
\end{equation}
\end{thm}
\begin{eg}
If $n=1$ and $m,r\in\mathbb{Z}_{\geq 0}$ then 
\begin{equation}\label{formn=1}
K_m(aq^{-r})=q^{-mr}\frac{(a^{-1};q)_{m+r}}{(a^{-1};q)_m(a^{-1};q)_r}
\end{equation}
by the explicit expression for $K_m(x)$ from Example \ref{n=1example}. The right hand side of
\eqref{formn=1} is manifestly
invariant under the interchange of $m$ and $r$.
\end{eg}
\begin{proof}
We divide the proof of the theorem in several steps.\\
{\bf Step 1.} If $K_u(a\widetilde{v})=K_v(a\widetilde{u})$ for all $v\in\mathbb{Z}^n$ then
$K_{s_iu}(a\widetilde{v})=K_v(a\widetilde{s_iu})$ for $v\in\mathbb{Z}^n$ and $1\leq i<n$.\\
{\bf Proof of step 1.} Writing out the formula from part 1 of Proposition \ref{PROPactionsame} gives
\begin{equation}\label{onex}
\begin{split}
\frac{(t-1)\widetilde{v}_i}{(\widetilde{v}_i-\widetilde{v}_{i+1})}K_u(a\widetilde{v})+
&\Bigl(\frac{\widetilde{v}_i-t\widetilde{v}_{i+1}}
{\widetilde{v}_i-\widetilde{v}_{i+1}}\Bigr)K_u(a\widetilde{s_{n-i}v})\\
=&\frac{(t-1)\overline{u}_i}{(\overline{u}_i-\overline{u}_{i+1})}K_u(a\widetilde{v})+
\Bigl(\frac{\overline{u}_i-t\overline{u}_{i+1}}{\overline{u}_i-\overline{u}_{i+1}}\Bigr)K_{s_iu}(a\widetilde{v}).
\end{split}
\end{equation}
Replacing in \eqref{onex} the role of $u$ and $v$ and replacing $i$ by $n-i$ we get
\begin{equation}\label{twox}
\begin{split}
\frac{(t-1)\widetilde{u}_{n-i}}{(\widetilde{u}_{n-i}-\widetilde{u}_{n+1-i})}K_v(a\widetilde{u})+
&\Bigl(\frac{\widetilde{u}_{n-i}-t\widetilde{u}_{n+1-i}}
{\widetilde{u}_{n-i}-\widetilde{u}_{n+1-i}}\Bigr)K_v(a\widetilde{s_{i}u})\\
=&\frac{(t-1)\overline{v}_{n-i}}{(\overline{v}_{n-i}-\overline{v}_{n+1-i})}K_v(a\widetilde{u})+
\Bigl(\frac{\overline{v}_{n-i}-t\overline{v}_{n+1-i}}{\overline{v}_{n-i}-\overline{v}_{n+1-i}}\Bigr)K_{s_{n-i}v}(a\widetilde{u}).
\end{split}
\end{equation}

Suppose that $s_{n-i}v=v$. Then $\overline{v}_{n-i}=t\overline{v}_{n+1-i}$ by the second part of Lemma \ref{LEMtech}. Since $\widetilde{v}=t^{1-n}w_0\overline{v}^{-1}$, i.e. $\widetilde{v}_i=t^{1-n}\overline{v}_{n+1-i}^{-1}$, we then also have $\widetilde{v}_i=t\widetilde{v}_{i+1}$.
It then follows by a direct computation that \eqref{onex} reduces to
$K_{s_iu}(a\widetilde{v})=K_u(a\widetilde{v})$ and \eqref{twox} to
$K_v(a\widetilde{s_iu})=K_v(a\widetilde{u})$ if $s_{n-i}v=v$. 

We now use these observations to prove step 1. Assume that $K_u(a\widetilde{v})=
K_v(a\widetilde{u})$ for all $v$. We have to show that 
$K_{s_iu}(a\widetilde{v})=K_v(a\widetilde{s_iu})$ for all $v$. It is trivially true if
$s_iu=u$, so we may assume that $s_iu\not=u$. 
Suppose that $v$ satisfies $s_{n-i}v=v$. Then it follows from the previous paragraph that
\[
K_{s_iu}(a\widetilde{v})=K_u(a\widetilde{v})=K_v(a\widetilde{u})=K_{v}(a\widetilde{s_iu}).
\]
If $s_{n-i}v\not=v$ then \eqref{onex} and the induction hypothesis can be used to write
$K_{s_iu}(a\widetilde{v})$ as an explicit linear combination of $K_v(a\widetilde{u})$ and 
$K_{s_{n-i}v}(a\widetilde{u})$. Then \eqref{twox} can be used to rewrite the term involving 
$K_{s_{n-i}v}(a\widetilde{u})$ as an explicit linear combination of $K_v(a\widetilde{u})$ and
$K_v(a\widetilde{s_iu})$. Hence we obtain an explicit expression of $K_{s_iu}(a\widetilde{v})$
as linear combination of $K_v(a\widetilde{u})$ and $K_v(a\widetilde{s_iu})$, which turns out to
reduce to
$K_{s_iu}(a\widetilde{v})=K_v(a\widetilde{s_iu})$ after a direct computation.\\
{\bf Step 2.} $K_0(a\widetilde{v})=1=K_v(a\widetilde{0})$ for all $v\in\mathbb{Z}^n$.\\
{\bf Proof of step 2.}
Clearly $K_0(x)=1$ and $K_v(a\widetilde{0})=K_v(a\tau)=1$ for $v\in\mathbb{Z}^n$ by Lemma \ref{LEMstep2}.\\
{\bf Step 3.} $K_\alpha(a\widetilde{v})=K_v(a\widetilde{\alpha})$ for $v\in \mathbb{Z}^n$ and 
$\alpha\in\mathcal{C}_n$.\\
{\bf Proof of step 3.} We prove it by induction. It is true for $\alpha=0$ by step 2. 
Let $m\in\mathbb{Z}_{>0}$ and suppose that $K_\gamma(a\widetilde{v})=
K_v(a\widetilde{\gamma})$
for $v\in\mathbb{Z}^n$ and $\gamma\in\mathcal{C}_n$ with $|\gamma|<m$. Let 
$\alpha\in\mathcal{C}_n$ with $|\alpha|=m$. 

We need to show
that $K_{\alpha}(a\widetilde{v})=K_v(a\widetilde{\alpha})$ for all $v\in\mathbb{Z}^n$.
By step 1 we may assume without loss of generality that $\alpha_n>0$. Then $\gamma:=\alpha^\sharp\in\mathcal{C}_n$ satisfies $|\gamma|=m-1$,
and $\alpha=\gamma^\natural$.
Furthermore, note that we have the formula
\begin{equation}\label{three1}
(a\overline{v}_1^{-1}-1)K_u(a\widetilde{v^\natural})=(a\overline{u}_1^{-1}-1)K_{u^\natural}(a\widetilde{v})
\end{equation}
for all $u,v\in\mathbb{Z}^n$, which follows by 
writing out the formula from part 3 of Lemma \ref{PROPactionsame}.
Hence we obtain
\begin{equation*}
\begin{split}
K_\alpha(a\widetilde{v})=K_{\gamma^\natural}(a\widetilde{v})&=\frac{(a\overline{v}_1^{-1}-1)}{(a\overline{\gamma}_1^{-1}-1)}K_\gamma(a\widetilde{v^\natural})\\
&=\frac{(a\overline{v}_1^{-1}-1)}{(a\overline{\gamma}_1^{-1}-1)}K_{v^\natural}(a\widetilde{\gamma})=K_v(a\widetilde{\gamma^\natural})=K_v(a\widetilde{\alpha}),
\end{split}
\end{equation*}
where we used the induction hypothesis for the third equality and \eqref{three1} for the second and fourth equality. This proves the induction step.\\
{\bf Step 4.} $K_u(a\widetilde{v})=K_v(a\widetilde{u})$ for all $u,v\in\mathbb{Z}^n$.\\
{\bf Proof of step 4.} Fix $u,v\in\mathbb{Z}^n$. Let $m\in\mathbb{Z}_{\geq 0}$ such that 
$u+(m^n)\in\mathcal{C}_n$. Note that $q^m\widetilde{v}=\widetilde{v-(m^n)}$
and $q^{-m}\widetilde{u}=\widetilde{u+(m^n)}$. Then
\begin{equation*}
\begin{split}
K_u(a\widetilde{v})&=A_m(a\widetilde{v};u)K_{u+(m^n)}(q^ma\widetilde{v})\\
&=A_m(a\widetilde{v};u)K_{u+(m^n)}\bigl(a(\widetilde{v-(m^n)})\bigr)\\
&=A_m(a\widetilde{v};u)K_{v-(m^n)}\bigl(a(\widetilde{u+(m^n)})\bigr)\\
&=A_m(a\widetilde{v};u)K_{v-(m^n)}(q^{-m}a\widetilde{u})=
A_m(a\widetilde{v};u)A_m(q^{-m}a\widetilde{u};v-(m^n))K_v(a\widetilde{u}),
\end{split}
\end{equation*}
where we used step 3 in the third equality. The result now follows from the fact that
\[
A_m(a\widetilde{v};u)A_m(q^{-m}a\widetilde{u};v-(m^n))=1,
\]
which follows by a straightforward computation using 
\eqref{tildeversusnormal}.
\end{proof}
\section{Some applications of duality}

\subsection{Non-symmetric Macdonald polynomials}

Recall that the (monic) non-symmetric Macdonald polynomial $E_\alpha(x)$ of degree $\alpha$
is the top homogeneous component of $G_\alpha(x)$, i.e.
\[
E_\alpha(x)=\lim_{a\rightarrow\infty} a^{-|\alpha|}G_\alpha(ax),\qquad \alpha\in\mathcal{C}_n.
\]
The normalized non-symmetric Macdonald polynomials are
\[
\overline{K}_\alpha(x):=\lim_{a\rightarrow\infty} K_\alpha(ax)=\frac{E_\alpha(x)}{E_\alpha(\tau)},\qquad \alpha\in\mathcal{C}_n.
\]
We write $\overline{K}\in\mathcal{F}^+_{\mathbb{F}[x]}$
for the resulting map $\alpha\mapsto \overline{K}_\alpha$.
Taking limits in Lemma \ref{LEMactiontransfer} we get
\begin{lem}\label{LEMtransferactionusual}
We have for $1\leq i< n$ and $1\leq j\leq n$,
\begin{enumerate}
\item[{\bf 1.}] $H_i\overline{K}=\widehat{H}_i\overline{K}$. 
\item[{\bf 2.}] $\xi_j\overline{K}=\widehat{x}_j^{-1}\overline{K}$.
\item[{\bf 3.}] $x_n\Delta\overline{K}=t^{1-n}\widehat{x}_1^{-1}\widehat{\Delta}^{-1}\overline{K}$.
\end{enumerate}
\end{lem}
Note that
\[
(x_n\Delta)^nf(x)=\Bigl(\prod_{i=1}^nx_i\Bigr)f(q^{-1}x).
\]
Then repeated application of part 3 of Lemma \ref{LEMtransferactionusual} shows that for $\alpha\in\mathcal{C}_n$,
\begin{equation}\label{relweightshom}
\begin{split}
E_\alpha(x)&=\frac{E_{\alpha+(1^n)}(x)}{x_1\cdots x_n},\\
\overline{K}_\alpha(x)&=q^{|\alpha|}t^{(1-n)n}\Bigl(\prod_{i=1}^n(\overline{\alpha}_ix_i)^{-1}\Bigr)
\overline{K}_{\alpha+(1^n)}(x).
\end{split}
\end{equation}
As is well known and already noted in Section \ref{s2}, the first equality allows to relate the non-symmetric
Macdonald polynomials $E_v(x):=E_v(x;q,t)\in\mathbb{F}[x^{\pm 1}]$ for arbitrary $v\in\mathbb{Z}^n$ to those labeled by compositions through the formula
\[
E_v(x)=\frac{E_{v+(m^n)}(x)}{(x_1\cdots x_n)^m}.
\]
The second formula of \eqref{relweightshom} can now be used to explicitly define the normalized non-symmetric Macdonald polynomials for degrees $v\in\mathbb{Z}^n$.
\begin{defi}
Let $v\in\mathbb{Z}^n$ and $m\in\mathbb{Z}_{\geq 0}$ such that $v+(m^n)\in\mathcal{C}_n$.
Then $\overline{K}_v(x):=\overline{K}_v(x;q,t)\in\mathbb{F}[x^{\pm 1}]$ is defined by
\begin{equation*}
\overline{K}_v(x):=q^{m|v|}t^{(1-n)nm}\Bigl(\prod_{i=1}^n(\overline{v}_ix_i)^{-m}\Bigr)
\overline{K}_{v+(m^n)}(x).
\end{equation*}
\end{defi}
Using
\[
\lim_{a\rightarrow\infty}A_m(ax;v)=q^{-m^2n}t^{(1-n)nm}\prod_{i=1}^n(\overline{v}_ix_i)^{-m}
\]
and the definitions of $G_v(x)$ and $K_v(x)$ it follows that
\begin{equation*}
\begin{split}
\lim_{a\rightarrow\infty}a^{-|v|}G_v(ax)&=E_v(x),\\
\lim_{a\rightarrow\infty}K_v(ax)&=\overline{K}_v(x)
\end{split}
\end{equation*}
for all $v\in\mathbb{Z}^n$, so in particular
\[
\overline{K}_v(x)=\frac{E_v(x)}{E_v(\tau)}\qquad \forall\, v\in\mathbb{Z}^n.
\]
Lemma \ref{LEMtransferactionusual} holds true for the extension of $\overline{K}$
to the map $\overline{K}\in\mathcal{F}_{\mathbb{F}[x^{\pm 1}]}$
defined by $v\mapsto \overline{K}_v$ ($v\in\mathbb{Z}^n$).
Taking the limit in Theorem \ref{THMsymmetry} we obtain the well known duality \cite{ChPap}
of the Laurent polynomial versions
of the normalized non-symmetric Macdonald polynomials.
\begin{cor}
For all $u,v\in\mathbb{Z}^n$,
\[
\overline{K}_u(\widetilde{v})=\overline{K}_v(\widetilde{u}).
\]
\end{cor}
\subsection{$O$-polynomials}
We now show that the duality of the non-symmetric interpolation Macdonald polynomials (Theorem \ref{THMsymmetry}) directly implies the existence of the $O$-polynomials $O_\alpha$ (which is the nontrivial part of the proof of \cite[Thm. 1.2]{S-bin}),
and that it provides an explicit expression for $O_\alpha$ in terms of the non-symmetric interpolation Macdonald polynomial $K_\alpha$.

\begin{prop}\label{propO}
For all $\alpha\in\mathcal{C}_n$ we have
\[
O_\alpha(x)=K_\alpha(t^{1-n}aw_0x).
\]
\end{prop}
\begin{proof}
The polynomial $\widetilde{O}_\alpha(x):=K_\alpha(t^{1-n}aw_0x)$ is of degree at most $|\alpha|$ and 
\[
\widetilde{O}_\alpha(\overline{\beta}^{-1})=K_\alpha(t^{1-n}aw_0\overline{\beta}^{-1})=K_\alpha(a\widetilde{\beta})=
K_\beta(a\widetilde{\alpha})
\]
for all $\beta\in\mathcal{C}_n$ by \eqref{tildeversusnormal} and Theorem \ref{THMsymmetry}. Hence $\widetilde{O}_\alpha=O_\alpha$.
\end{proof}

\subsection{Okounkov's duality}
Write $F[x]^{S_n}$ for the symmetric polynomials in $x_1,\ldots,x_n$
with coefficients in a field $F$. Write $C_+:=\sum_{w\in S_n}H_w$. The symmetric interpolation Macdonald polynomial $R_\lambda(x)\in\mathbb{F}[x]^{S_n}$ is the multiple of $C_+G_\lambda$ such that the coefficient of $x^\lambda$ is one
(see, e.g., \cite{S-int}).
We write
\[
K_\lambda^+(x):=\frac{R_\lambda(x)}{R_\lambda(a\tau)}\in\mathbb{K}[x]^{S_n}
\]
for the normalized symmetric interpolation Macdonald polynomial. Then 
\begin{equation}\label{relsymm}
C_+K_\alpha(x)=
\Bigl(\sum_{w\in S_n}t^{\ell(w)}\Bigr)K_{\alpha_+}^+(x)
\end{equation}
for $\alpha\in\mathcal{C}_n$. Okounkov's \cite[\S 2]{Ok} duality result now reads as follows.
\begin{thm}\label{Okounkovthm}
For partitions $\lambda, \mu\in\mathcal{P}_n$ we have
\[
K_\lambda^+(a\overline{\mu}^{-1})=K_\mu^+(a\overline{\lambda}^{-1}).
\]
\end{thm}
Let us derive Theorem \ref{Okounkovthm} as consequence of Theorem \ref{THMsymmetry}.
Write $\widehat{C}_+=\sum_{w\in S_n}\widehat{H}_w$, with
$\widehat{H}_w:=\widehat{H}_{i_1}\cdots \widehat{H}_{i_r}$ for a reduced expression $w=s_{i_1}\cdots s_{i_r}$. Write $f_\mu\in\mathcal{F}_\mathbb{K}$ for the function $f_\mu(u):=K_u(a\widetilde{\mu})$ ($u\in\mathbb{Z}^n$). Then
\begin{equation}\label{tussenstap}
\Bigl(\sum_{w\in S_n}t^{\ell(w)}\Bigr)K_{\lambda}^+(a\widetilde{\mu})=(C_+K_\lambda)(a\widetilde{\mu})=(\widehat{C}_+f_\mu)(\lambda)
\end{equation}
by part 1 of Proposition \ref{PROPactionsame}. The duality \eqref{Fsymmetry} of $K_u$ and
\eqref{tildeversusnormal}
imply that
\begin{equation}\label{tussenstap2}
f_\mu(u)=K_\mu(a\widetilde{u})=\bigl(Jw_0K_\mu(t^{1-n}x)\bigr)|_{x=a^{-1}\overline{u}}
\end{equation}
with $(Jf)(x):=f(x_1^{-1},\ldots,x_n^{-1})$ for $f\in\mathbb{K}(x)$. A direct computation shows
that
\begin{equation}\label{Iw0}
JH_iJ=(H_i^\circ)^{-1},\qquad w_0H_iw_0=(H_{n-i}^\circ)^{-1}
\end{equation}
for $1\leq i<n$. In particular, $Jw_0C_+=C_+Jw_0$. Combined with Remark \ref{HatRelation}
we conclude that
\[
(\widehat{C}_+f_\mu)(\lambda)=\bigl(Jw_0C_+K_\mu(t^{1-n}x)\bigr)|_{x=a^{-1}\overline{\lambda}}.
\]
By \eqref{relsymm} and \eqref{tildeversusnormal}
this simplifies to
\[
(\widehat{C}_+f_\mu)(\lambda)=\Bigl(\sum_{w\in S_n}t^{\ell(w)}\Bigr)K_{\mu}^+(a\widetilde{\lambda}).
\]
Returning to \eqref{tussenstap} we conclude that $K_\lambda^+(a\widetilde{\mu})=
K_\mu^+(a\widetilde{\lambda})$.
Since $K_\lambda^+$ is symmetric 
we obtain from \eqref{tildeversusnormal} that
\[
K_\lambda^+(a\overline{\mu}^{-1})=K_\mu^+(a\overline{\lambda}^{-1}),
\]
which is Okounkov's duality result.

\subsection{A primed version of duality}

We first derive the following twisted version of the duality of the non-symmetric interpolation
Macdonald polynomials (Theorem \ref{THMsymmetry}).

\begin{lem}\label{CORw0symmetry}
For $u,v\in\mathbb{Z}^n$ we have
\begin{equation}\label{w0twistedsymmetry}
\bigl(H_{w_0} K_u\bigr)(a\widetilde{v})=\bigl(H_{w_0} K_v\bigr)(a\widetilde{u}).
\end{equation}
\end{lem}
\begin{proof}
We proceed as in the previous subsection. Set $f_v(u):=K_u(a\widetilde{v})$ for $u,v\in\mathbb{Z}^n$.
By part 1 of Proposition \ref{PROPactionsame},
\[
\bigl(H_{w_0}K_u\bigr)(a\widetilde{v})=\bigl(\widehat{H}_{w_0}f_v\bigr)(u).
\]
Since $f_v(u)=\bigl(Iw_0K_v\bigr)(a^{-1}t^{n-1}\overline{u})$
by \eqref{tildeversusnormal}, Remark \ref{HatRelation} implies that
\[
\bigl(\widehat{H}_{w_0}f_v\bigr)(u)=
\bigl(H_{w_0}Jw_0K_v\bigr)(a^{-1}t^{n-1}\overline{u}).
\]
Now $H_{w_0}Jw_0=Jw_0H_{w_0}$ by \eqref{Iw0}, hence
\[
\bigl(\widehat{H}_{w_0}f_v\bigr)(u)=
\bigl(Jw_0H_{w_0}K_v\bigr)(a^{-1}t^{n-1}\overline{u})=
(H_{w_0}K_v)(a\widetilde{u}),
\]
which completes the proof.
\end{proof}

Recall from Theorem \ref{THMprime} that
\[
G_\beta^\prime(x)=t^{(1-n)|\beta|+I(\beta)}\Psi G_\beta^\circ(t^{n-1}x)
\]
with $\Psi:=w_0H_{w_0}^\circ$. We define normalized versions by
\[
K_\beta^\prime(x):=\frac{G_\beta^\prime(x)}{G_\beta^\prime(a^{-1}\tau)}=t^{\ell(w_0)}\Psi K_\beta^\circ(t^{n-1}x),\qquad \beta\in\mathcal{C}_n,
\]
with $K_v^\circ:=\iota(K_v)$ for $v\in\mathbb{Z}^n$ (the second formula follows from
Lemma \ref{CORevalprime}).
More generally, we define for $v\in\mathbb{Z}^n$,
\begin{equation}\label{defprime}
K_v^\prime(x):=t^{\ell(w_0)}\Psi K_v^\circ(t^{n-1}x).
\end{equation}
We write $K^\prime: \mathbb{Z}^n\rightarrow\mathbb{K}(x)$ for the map
$v\mapsto K_v^\prime$ ($v\in\mathbb{Z}^n$).  Since $H_i\Psi=\Psi H_i^\circ$, part 1 of Proposition \ref{PROPactionsame} gives 
$H_iK^\prime=\widehat{H}_i^\circ K^\prime$. Considering the action of $((x_n-1)\Delta^\circ)^n$ on $K_\beta^\prime(x)$ we get, using the fact that
$((x_n-1)\Delta^\circ)^n$ commutes with $\Psi$ and part 3 of Proposition \ref{PROPactionsame},
\[
K_v^\prime(x)=\Bigl(\prod_{i=1}^n\frac{(1-a^{-1}\overline{v}_i)}{(1-q^{-1}x_i)}\Bigr)K_{v+(1^n)}^\prime(q^{-1}x),
\]
in particular
\[
K_v^\prime(x)=\Bigl(\prod_{i=1}^n\frac{\bigl(a^{-1}\overline{v}_i;q\bigr)_m}{\bigl(q^{-m}x_i;q\bigr)_m}\Bigr)K_{v+(m^n)}^\prime(q^{-m}x).
\]
\begin{eg}
For $n=1$ we have $K_v^\prime(x)=K_v^\circ(x)$
for $v\in\mathbb{Z}$, hence
\begin{equation*}
\begin{split}
K_{-m}^\prime(x)&=\frac{\bigl(q^{-1}a^{-1};q^{-1}\bigr)_m}{\bigl(q^{-1}x;q^{-1}\bigr)_m}=(ax)^{-m}\frac{\bigl(qa;q\bigr)_m}{\bigl(qx^{-1};q\bigr)_m},\\
K_{m}^\prime(x)&=(ax)^m\frac{\bigl(x^{-1};q^{-1}\bigr)_m}{\bigl(a;q^{-1}\bigr)_m}=\frac{\bigl(x;q\bigr)_m}{\bigl(a^{-1};q\bigr)_m}
\end{split}
\end{equation*}
for $m\in\mathbb{Z}_{\geq 0}$ by Example \ref{n=1example}.
\end{eg}

\begin{prop}
For all $u,v\in\mathbb{Z}^n$ we have
\[
K_v^\prime(a^{-1}\overline{u})=K_u^\prime(a^{-1}\overline{v}).
\]
\end{prop}
\begin{proof}
Note that
\[
K_v^\prime(a^{-1}\overline{u})=t^{\ell(w_0)}\Psi K_v^\circ(t^{n-1}x)|_{x=a^{-1}\overline{u}}=
t^{\ell(w_0)}\bigl(H_{w_0}^\circ K_v^\circ\bigr)(a^{-1}\widetilde{u}^{-1})
\]
by \eqref{tildeversusnormal}.
By \eqref{w0twistedsymmetry} the right hand side is invariant under the interchange of $u$
and $v$.
\end{proof}

\subsection{Binomial formula and dual binomial formula}\label{Osection}
In \cite{S-bin} the existence and uniqueness of $O_\alpha$ was used to prove the following binomial theorem \cite[Thm. 1.3]{S-bin}. Define for $\alpha, \beta\in\mathcal{C}_n$ the generalized binomial coefficient by
\begin{equation}\label{binomcoeff}
\left[ \begin{matrix} \alpha\\ \beta\end{matrix}\right]_{q,t}:=\frac{G_\beta(\overline{\alpha})}{G_\beta(\overline{\beta})}.
\end{equation}
Applying the automorphism $\iota$ of $\mathbb{F}$ to \eqref{binomcoeff} we get
\[
\left[\begin{matrix} \alpha\\ \beta\end{matrix}\right]_{q^{-1},t^{-1}}=\frac{G_\beta^\circ(\overline{\alpha}^{-1})}{G_\beta^\circ(\overline{\beta}^{-1})}.
\]
\begin{thm}
For $\alpha,\beta\in\mathcal{C}_n$ we have the binomial formula
\begin{equation}\label{binomialformula}
K_\alpha(ax)=\sum_{\beta\in\mathcal{C}_n} a^{|\beta|}\left[\begin{matrix} \alpha\\ \beta\end{matrix}\right]_{q^{-1},t^{-1}}\frac{G_\beta^\prime(x)}{G_\beta(a\tau)}.
\end{equation}
\end{thm}
\begin{rema}
{\bf 1.} Note that the sum in \eqref{binomialformula} is finite, since the generalized
binomial coefficient \eqref{binomcoeff} is zero
unless $\beta\subseteq\alpha$,
with $\beta\subseteq\alpha$ meaning $\beta_i\leq\alpha_i$ for $i=1,\ldots,n$.\\
{\bf 2.} By Corollary \ref{CORrelev} and \eqref{defprime} the binomial formula \eqref{binomialformula} can be alternatively written as
\begin{equation}\label{binomialformulanew}
\begin{split}
K_\alpha(ax)&=\sum_{\beta\in\mathcal{C}_n}\tau_\beta^{-1}\left[\begin{matrix} \alpha\\ \beta\end{matrix}\right]_{q^{-1},t^{-1}}K_\beta^\prime(x)\\
&=\sum_{\beta\in\mathcal{C}_n}\frac{K_\beta^\circ(\overline{\alpha}^{-1})K_\beta^\prime(x)}{\tau_\beta K_\beta^\circ(\overline{\beta}^{-1})}\\
&=t^{\ell(w_0)}\sum_{\beta\in\mathcal{C}_n}\frac{K_\beta^\circ(\overline{\alpha}^{-1})\Psi K_\beta^\circ (t^{n-1}x)}{\tau_\beta K_\beta^0(\overline{\beta}^{-1})}
\end{split}
\end{equation}
with $\Psi=w_0H_{w_0}^\circ$ (note that the dependence on $a$ in the right hand side of \eqref{binomialformulanew} is through the normalization factors of the interpolation polynomials
$K_\beta^\circ(x)$ and $K_\beta^\prime(x)$).\\
{\bf 3.} The binomial formula \eqref{binomialformula} and Theorem \ref{THMprime} imply the twisted duality \eqref{w0twistedsymmetry} of $K_\alpha$ as follows.
By the identity $H_{w_0}\Psi=w_0$ the binomial formula \eqref{binomialformulanew} implies the finite expansion
\[
\bigl(H_{w_0}K_\alpha\bigr)(ax)=t^{\ell(w_0)}\sum_\beta \frac{K_\beta^\circ(\overline{\alpha}^{-1})K_\beta^\circ(t^{n-1}w_0x)}{\tau_\beta K_\beta^\circ(\overline{\beta}^{-1})}.
\]
Substituting $x=\widetilde{\gamma}$ and using \eqref{tildeversusnormal}
 we obtain
\begin{equation*}
\bigl(H_{w_0}K_\alpha\bigr)(a\widetilde{\gamma})=\sum_{\beta\in\mathcal{C}_n}\frac{K_\beta^\circ(\overline{\alpha}^{-1})K_\beta^\circ(\overline{\gamma}^{-1})}{\tau_\beta K_\beta^\circ(\overline{\beta}^{-1})}.
\end{equation*}
The right hand side is manifestly invariant under interchanging $\alpha$ and $\gamma$,
which is equivalent to twisted duality \eqref{w0twistedsymmetry}.
\end{rema}

In \cite[\S 4]{LRW} it is remarked that an explicit identity relating $G_\alpha^\prime$ and $G_\alpha$ 
is needed to provide a proof of the dual binomial formula \cite[Thm. 4.4]{LRW} as a direct consequence of the binomial formula \eqref{binomialformula}.
We show here that Theorem \ref{THMprime} is providing the required identity. Instead of Theorem \ref{THMprime} we use its normalized version, 
encoded by \eqref{defprime}.

The dual binomial formula \cite[Thm. 4.4]{LRW} in our notations reads as follows. 
\begin{thm}
\label{THMlrw} 
For all $\alpha\in\mathcal{C}_n$ we have
\begin{equation}\label{dualbinomialformula}
K_\alpha^\prime(x)=\sum_{\beta\in\mathcal{C}_n}\tau_\beta \left[ \begin{matrix} \alpha\\ \beta\end{matrix}\right]_{q,t}K_\beta(ax).
\end{equation}
\end{thm}
The starting point of the alternative proof of \eqref{dualbinomialformula} is the binomial formula 
in the form
\[
K_\alpha(ax)=t^{\ell(w_0)}\sum_{\beta\in\mathcal{C}_n}\frac{G_\beta^\circ(\overline{\alpha}^{-1})\Psi K_\beta^\circ (t^{n-1}x)}{\tau_\beta G_\beta^\circ(\overline{\beta}^{-1})},
\]
see \eqref{binomialformulanew}. Replace $(a,x,q,t)$ by $(a^{-1},at^{n-1}x,q^{-1},t^{-1})$ and act by $w_0H_{w_0}$ on both sides.
Since $w_0H_{w_0}\Psi=\textup{Id}$ we obtain
\[
\Psi K_\alpha^\circ(t^{n-1}x)=t^{-\ell(w_0)}\sum_\beta \tau_\beta \left[ \begin{matrix} \alpha\\ \beta\end{matrix}\right]_{q,t}K_\beta(ax).
\]
Now use \eqref{defprime} to complete the proof of \eqref{dualbinomialformula}.

\begin{rema}
It follows from this proof of \eqref{dualbinomialformula} that the dual binomial formula \eqref{dualbinomialformula} can be rewritten as
\begin{equation}\label{dualbinomialformulanew}
\Psi K_\alpha^\circ(t^{n-1}x)=t^{-\ell(w_0)}\sum_{\beta}\frac{\tau_\beta K_\beta(\overline{\alpha})K_\beta(ax)}{K_\beta(\overline{\beta})}.
\end{equation}
\end{rema}

As observed in \cite[(4.11)]{LRW}, the binomial and dual binomial formula directly imply the orthogonality relations 
\[
\sum_{\beta\in\mathcal{C}_n}\frac{\tau_\beta}{\tau_\alpha}\left[\begin{matrix} \alpha\\ \beta\end{matrix}\right]_{q,t}\left[\begin{matrix} \beta\\ \gamma\end{matrix}\right]_{q^{-1},t^{-1}}=
\delta_{\alpha,\gamma}.
\]
Since $\left[\begin{matrix} \delta\\ \epsilon\end{matrix}\right]_{q, t}=0$ unless
$\delta\supseteq\epsilon$, the terms in the sum are zero unless $\gamma\subseteq\beta\subseteq\alpha$.




\begin{thebibliography}{99}
\bibitem{ChPap} I. Cherednik, {\it Nonsymmetric Macdonald polynomials}, Int. Math. Res. Not.
IMRN {\bf 1995}, no. 10, 483--515.
\bibitem{Ch} I. Cherednik, {\it Double affine Hecke algebras}, London Math. Society Lecture Note Series {\bf 319}, Cambridge University Press, Cambridge (2005).
\bibitem{Kn} F. Knop, {\it Symmetric and nonsymmetric quantum Capelli polynomials}, Comment. Math. Helv. {\bf 72} (1997), 84--100.
\bibitem{KS1} F. Knop, S. Sahi, \textit{Difference equations and symmetric
polynomials defined by their zeros}. Int. Math. Res. Not. IMRN {\bf 1996},
no. 10, 473--486.
\bibitem{KoS1} B. Kostant, S. Sahi,\emph{\ }\textit{The Capelli identity,
tube domains, and the generalized Laplace transform}. Adv. Math. 87 (1991),
no. 1, 71--92.
\bibitem{KoS2} B. Kostant, S. Sahi, \emph{\ }\textit{Jordan algebras and
Capelli identities}. Invent. Math. 112 (1993), no. 3, 657--664.
\bibitem{LRW} A. Lascoux, E.M. Rains, S.O. Warnaar, {\it Nonsymmetric interpolation Macdonald polynomials and $\mathfrak{gl}_n$ basic hypergeometric
series}, Transform. Groups {\bf 14}, no. 3 (2009), 613--647.
\bibitem{Mac} I.G. Macdonald, Symmetric functions and Hall polynomials,
Clarendon Press, Oxford 1995 (2nd edition).
\bibitem{Ok} A. Okounkov, {\it Binomial formula for Macdonald polynomials and applications}, Math Res. Lett. {\bf 4}, 533--553 (1997). 
\bibitem{S-inv} S. Sahi, \textit{The spectrum of certain invariant
differential operators associated to a Hermitian symmetric space}. Lie
theory and geometry, 569--576, Progr. Math., 123, Birkhauser Boston, Boston,
MA, 1994.
\bibitem{S-int} S. Sahi, \textit{Interpolation, integrality, and a
generalization of Macdonald's polynomials}, Int. Math. Res. Not. IMRN 
\textbf{1996}, no. 10, 457--471.
\bibitem{S-bin} S. Sahi, \textit{The binomial formula for nonsymmetric
Macdonald polynomials}, Duke Math. J. \textbf{94}, no. 3 (1998), 465--477.
\end{thebibliography}
\end{document}